\documentclass{amsart}
\usepackage{amssymb,amsmath,exscale,latexsym,amsthm,graphics,overpic,wasysym}
\usepackage{epsfig}    
\usepackage{epic}      
\usepackage{eepic}     
\usepackage[matrix,arrow,curve,cmtip]{xy} 
\usepackage{enumitem} 
\usepackage{color}
\usepackage{verbatim}
\usepackage{hyperref} 
\usepackage{ mathrsfs } 

\theoremstyle{plain}
        \newtheorem{theorem}{Theorem}
        \newtheorem*{theorem*}{Theorem}
        \newtheorem*{conj*}{Conjecture}
        \newtheorem*{FGAconj*}{Finite Global Curve Attractor Conjecture}
        \newtheorem{lemma}[theorem]{Lemma}
        \newtheorem{cor}[theorem]{Corollary}
        
        \newtheorem{prop*}{Proposition}

\theoremstyle{definition}

\theoremstyle{remark}
        \newtheorem*{remark}{Remark}
        \newtheorem*{remarks}{Remarks}

        \newtheorem*{claim}{Claim}

\numberwithin{equation}{section}

\newcommand{\id} {\operatorname{id}}
 
\renewcommand{\Im} {\operatorname{Im}}
\newcommand{\Qh}{\widehat{\mathbb{Q}}}
\newcommand{\R}{\mathbb{R}}      
\newcommand{\C}{\mathbb{C}}      
\newcommand{\N}{\mathbb{N}}      
\newcommand{\Q}{\mathbb{Q}}      
\newcommand{\Halb}{\mathbb{H}}   
\newcommand{\CDach}{\widehat{\mathbb{C}}}
\newcommand{\D}{\mathbb{D}}      

\renewcommand{\:}{\colon}   
\newcommand{\ra}{\rightarrow} 
\newcommand{\sub}{\subset}

\newcommand{\Sp}{S^2}

\newcommand{\inter}{\operatorname{int}}

\newcommand{\trivial}{\odot}

\newcommand{\Orb}{\mathcal{O}}

\newcommand{\X} {\mathcal{X}}

\newcommand{\W} {\mathcal{W}}

\newcommand{\NC} {\mathcal{N}}

\newcommand{\T}{\mathcal{T}}
\newcommand{\M}{\mathcal{M}}

\begin{document}

\title[Thurston's pullback map and invariant covers] {Thurston's pullback map, invariant covers, and the global dynamics on curves}
\author{Mario Bonk}
\address{Department of Mathematics, University of California, Los Angeles, CA 90095, USA}
\email{mbonk@math.ucla.edu}
\author{Mikhail Hlushchanka}
\address{Korteweg-de Vries Instituut voor Wiskunde, Universiteit van Amsterdam,  1090 GE \newline Amsterdam, The Netherlands}
\email{mikhail.hlushchanka@gmail.com}
\author{Russell Lodge}
\address{Department of Mathematical Sciences, Indiana State University, Terre Haute, IN 47809, USA}
\email{russell.lodge@indstate.edu}
\date{\today}


\thanks{The first author was  partially supported by NSF grant DMS-1808856. 
 The second author was  partially supported by NSF grant DMS-1440140, while the authors
  participated in a program hosted by the Mathematical Sciences
  Research Institute in Berkeley, California, during the Spring semester of 2022. The second author was also partially supported by the Marie Sk\l{}odowska-Curie Postdoctoral Fellowship under Grant No.\ 101068362. }

 \begin{abstract} 
 We consider rational maps $f$ on the Riemann sphere $\CDach$ with an $f$-invariant set $P\sub \CDach$ of four marked points containing the postcritical set of $f$.
 We show that the dynamics of the corresponding Thurston pullback map $\sigma_f$ on the
 completion $\overline{\T_P}$ of the associated Teichm\"uller space $\T_P$ with respect to the Weil--Petersson metric 
is easy to understand when $\overline{\T_P}$ admits a cover by sets  with good combinatorial and dynamical properties. In particular, the map $f$ has a finite global curve attractor in this case. Using a result by 
Eremenko and Gabrielov, we also show that if $P$ contains all critical points of $f$ and each point in $P$ is periodic, then such a cover of $\overline{\T_P}$ can be obtained 
from  a $\sigma_f$-invariant tessellation by ideal hyperbolic triangles.  

   \end{abstract}

\maketitle
\section{Introduction}
In this note we consider (orientation-preserving) branched covering maps 
$f\:S^2\ra S^2$  on a topological $2$-sphere $\Sp$. 
 We use $f^n$, $n\in \N$, to denote the $n$-th iterate of $f$ and $C_f$ to denote the set of critical points of $f$.  Then $f$ is called \emph{postcritically-finite} if its postcritical set 
 \[ P_f :=\bigcup_{n=1}^\infty f^{n}(C_f)
 \]
  is finite. A (\emph{marked}) \emph{Thurston map} is a branched covering map $f\colon (\Sp, P)\to (\Sp, P)$, where $P\subset \Sp$ is a finite set of marked points satisfying $P_f\subset P$ and $f(P)\subset P$; in particular, every Thurston map is postcritically-finite. We assume throughout that $|P|\geq 4$ and will be mainly interested in the case when $f$ is a rational map on the Riemann sphere $\CDach:=\C\cup\{\infty\}$ with $|P|=4$.

We denote by $\mathscr{C}_P$ the set of all isotopy classes $[\gamma]$ of essential (non-oriented) simple closed curves $\gamma$ in $\Sp\setminus P$. We set  $\overline{\mathscr{C}_P}:=\mathscr{C}_P\cup \{\trivial\}$, where $\trivial$ represents the isotopy classes of all non-essential curves in $\Sp\setminus P$.

The map $f$ induces a \emph{pullback relation} $\xleftarrow{f}$ on the set $\overline{\mathscr{C}_P}$: given $[\gamma]\in \mathscr{C}_P$ and a component $\delta$ of $f^{-1}(\gamma)$, we declare $[\gamma]\xleftarrow{f}[\delta]$ if $\delta$ is essential  and $[\gamma]\xleftarrow{f}\trivial$ if $\delta$ is non-essential; in addition, we also set $\trivial \xleftarrow{f} \trivial$. When $|P|=4$, all essential components of $f^{-1}(\gamma)$ are isotopic to each other, and so 
the relation $\xleftarrow{f}$ determines  a well-defined  \emph{pullback map} $\mu_f\: \overline{\mathscr{C}_P}\ra \overline{\mathscr{C}_P}$. Namely, we set $\mu_f([\gamma])=[\delta]$ if $f^{-1}(\gamma)$ has some essential component $\delta$ and $\mu_f([\gamma])= \trivial$, otherwise.

%
%

One motivation for the investigations in this paper is   the following conjecture.

\begin{FGAconj*} If $f\colon (\CDach, P)\to (\CDach, P)$ is a rational Thurston map with a hyperbolic orbifold, then the pullback relation $\xleftarrow{f}$  on curves has a finite global attractor, that is, there is a finite set $\mathcal{A}\subset \overline{\mathscr{C}_P}$ such that every orbit $[\gamma_0]\xleftarrow{f}[\gamma_1]\xleftarrow{f}[\gamma_2]\xleftarrow{f}\dots$ eventually lies in $\mathcal{A}$.
\end{FGAconj*}

The minimal set $\mathcal{A}\subset \overline{\mathscr{C}_P}$ satisfying the conjecture above is called the \emph{global attractor} of $\xleftarrow{f}$.

The Finite Global Curve Attractor Conjecture has recently been  confirmed for all postcritically-finite polynomial maps \cite{belk2022recognizing}. However, in the setting of non-polynomial rational maps only partial results are available \cite{pilgrim2012algebraic,lodge2012boundary,hlushchanka2019tischler,kelsey2019quadratic,bonk2021eliminating,smith2024curve}. By a very recent result of Bartholdi, Dudko, and Pilgrim it is now known when $|P|=4$ \cite{bartholdi2024correspondences}.  Their proof is by contradiction and gives no explicit way of finding the global attractor. Therefore, it remains valuable to explore alternative approaches that could lead to an identification of the attractor. The purpose of this paper is to provide some results in this direction.

The pullback relation on curves for a Thurston map $f\colon (\Sp, P)\to (\Sp, P)$ 
is closely related to the \emph{Thurston pullback map} $\sigma_f\colon \T_P\to \T_P$ on the Teichm\"uller space associated with $f$ \cite{koch2016pullback}. In fact, we will use this relation to establish  a sufficient condition for the existence of a finite global attractor of $\xleftarrow{f}$ in the  special case when $|P|=4$. To formulate our result,  we 
first fix some terminology and notation. 

When $|P|=4$, we may identify the Teichm\"uller space $\T_P$ with the upper half-plane $\Halb:=\{z\in \C: \Im(z)>0\}$
and the Weil--Petersson completion $\overline{\T_P}$ with $\Halb^*:=\Halb \cup \Qh$, where $\Qh:=\Q\cup \{\infty\}$. By  work of Selinger \cite{selinger2012thurston}, the Thurston pullback map $\sigma_f$ admits a continuous extension to the 
Weil--Petersson completion $\overline{\T_P}$, which we also denote by $\sigma_f$ for simplicity. We can now formulate our first result.

\begin{theorem}\label{thm-intro:FGA-tiling}
Let $f\colon (\CDach, P)\to (\CDach, P)$ be a rational Thurston map with $|P|=4$ and a  hyperbolic orbifold. Suppose that there is a cover $\mathcal{U}$ of the Weil--Petersson completion $\overline{\T_P}=\Halb^*$ by some of its subsets  such that the following conditions are true:
\begin{enumerate}[label=\normalfont{(\roman*)}]

\smallskip
 \item\label{item:i} every point in  $\T_P$ has a neighborhood that intersects only finitely many sets in $\mathcal{U}$,

\smallskip 
    \item\label{item:ia} every set in $\mathcal{U}$  contains at least one point in 
    $\T_P=\Halb$  and    at most finitely many points in $\partial \T_P = \widehat{\Q}$, 
   
  \smallskip
 \item\label{item:ii} for every set $T\in  \mathcal{U}$ there exists $T'\in \mathcal{U}$ such 
    such that $\sigma_f(T)\sub T'$.
\end{enumerate}
 Then the pullback relation $\xleftarrow{f}$ on curves has a finite global attractor.  
\end{theorem}

If a  cover $\mathcal{U}$ of  $\overline{\T_P}=\Halb^*$  has properties \ref{item:i}--\ref{item:ii}, then we say that it is {\em $\sigma_f$-invariant}.  Once such a cover exists, one can establish the existence of a finite global attractor of $\xleftarrow{f}$ by a very short argument and  easily exhibit a finite set that contains the global attractor (see the remarks after the proof of Theorem~\ref{thm-intro:FGA-tiling} in Section~\ref{sec:FGA-invariant-tiling}; there we also discuss how conditions~\ref{item:i} and \ref{item:ia}
can be relaxed). Since the dynamics of the Thurston pullback map is quite intricate in general, the existence of a $\sigma_f$-invariant  cover   is an  interesting result on its own.

Our argument for the existence of a finite global attractor in Theorem~\ref{thm-intro:FGA-tiling} relies on the pointwise convergence of the iterates  of $\sigma_f$ to the unique fixed point in the Teichm\"uller space $\T_P$. In this way, it differs conceptually  from previous results relying on algebraic  tools \cite{pilgrim2012algebraic, lodge2012boundary,kelsey2019quadratic} or combinatorial arguments \cite{hlushchanka2019tischler, belk2022recognizing, bonk2021eliminating} or an asymptotic analysis \cite{bartholdi2024correspondences}.


We do not know whether  Theorem~\ref{thm-intro:FGA-tiling} is applicable to all rational Thurston maps with four marked points and a hyperbolic orbifold. However, it does apply to some natural infinite families of such maps, such as rational Thurston maps with all critical points in the marked set and rational Thurston maps with a moduli space map (Corollary~\ref{cor:injective-X}). More specifically, we use a result of Eremenko and Gabrielov on rational functions with real critical points \cite{eremenko1999} to prove the existence of a $\sigma_f$-invariant cover for the former family.


%
%


\begin{theorem}\label{thm:inv-tiling-crit-points}
    Let $f\colon (\CDach, P)\to (\CDach, P)$ be a rational Thurston map with $|P|=4$ 
    and a hyperbolic orbifold. Suppose  that $C_f\subset P$ and that every point in $P$ is periodic. Then there exists a $\sigma_f$-invariant cover $\mathcal{U}$ of 
    $\overline{\T_P}$. 
  \end{theorem}  
    
 As we will see, the cover $\mathcal{U}$ can be described explicitly here. 
 Namely,  let  $\pi\colon \T_P\to \M_P$ denote the canonical universal covering map
 of the moduli space $\M_P$. There is a natural identification $\M_P$ with a thrice-punctured sphere. If 
 $G$ is   the circle in $\overline{\M_P}\cong \CDach$ passing through the punctures of $\M_P$, then the closures of the complementary components of $\pi^{-1}(G)$ in 
 $\overline{\T_P}=\Halb^*$ are ideal hyperbolic triangles and form the sets of a $\sigma_f$-invariant cover $\mathcal{U}$.

Our paper is structured as follows. First, we introduce the relevant concepts and notation in Section~\ref{sec:pullback-map}. Next, we provide the  (short) proof of Theorem~\ref{thm-intro:FGA-tiling} in Section~\ref{sec:FGA-invariant-tiling}. Finally, Section~\ref{sec:examples-and-applications} is devoted to discussions of various applications of this theorem. Here, Lemma~\ref{lem:inv_tiling} provides a key sufficient condition for the existence of a $\sigma_f$-invariant cover of the Weil--Petersson completion of  Teichm\"uller space. We use this lemma to derive the existence of $\sigma_f$-invariant covers for special families of rational maps in Sections~\ref{subsec:four-periodic-crit} and \ref{subsec:injective-X}. In Section~\ref{subsec:examples}, we give some explicit examples of rational maps to which our criterion applies. 

\subsection*{Acknowledgments} The authors would like to thank Dima Dudko, Daniel
Meyer, Kevin Pilgrim, Nikolai Prochorov, and Zachary Smith for various valuable discussion and comments.

\section{Thurston's pullback map}\label{sec:pullback-map}
We refer the reader to \cite{Douady1993,Buff2009Pullback} for general background on Thurston's theory of rational maps. 

In the following, let $f\colon (\Sp, P)\to (\Sp,P)$ be a Thurston map. The \emph{orbifold} $\Orb_f$ associated with $f$ is the topological orbifold with the underlying space $\Sp$ and cone points at every point $p \in P_f$ of order $\nu(p)$,
where $\nu(p)$ is the least common multiple of the local degrees of
the iterates $f^n$ at all points $q \in f^{-n}(\{p\})$, over all $n \ge 1$. 
The \emph{Euler characteristic} of $\Orb_f$ is the number
\[
  \chi(\Orb_f) = 2-\sum\limits_{p \in P_f}
  \left(1-\frac{1}{\nu(p)}\right).
\] 
The orbifold $\Orb_f$ is called \emph{hyperbolic} if  $\chi(\Orb_f)< 0$, and parabolic if $\chi(\Orb_f)= 0$ (note that $\chi(\Orb_f)\leq 0$ for all Thurston maps $f$).
For more details see, e.g., \cite[Chapter 2]{BM17}.

We denote by $\T_P$ the \emph{Teichm\"uller space} of the marked sphere $(\Sp,P)$. It can be defined as the set of all equivalence classes $[\phi]$ of orientation-preserving homeomorphisms $\phi$  from $(\Sp,P)$ to the Riemann sphere $\CDach$.  Here two such homeomorphism  $\phi_1,\phi_2\colon (\Sp,P)\to\CDach$ are equivalent if and only if there is a M\"obius transformation $M$ such that $M\circ\phi_1$ is isotopic to $\phi_2$ rel.\ $P$.

 We denote by $\M_P$ the corresponding \emph{moduli space} of $(\Sp,P)$, which is the space of all injections from $P$ to $\CDach$ modulo postcomposition with M\"obius transformations. The Teichm\"uller and moduli spaces are complex manifolds of dimension $|P|-3$. We denote by $\pi\colon \T_P\to\M_P$ the canonical holomorphic universal covering map given by $[\phi]\mapsto[\phi|_{P}]$. Let $G_P$ be the \emph{pure mapping class group} of $(\Sp, P)$, that is, it is the group of all orientation-preserving homeomorphisms $h\colon (\Sp,P) \to (\Sp,P)$ with $h|_{P}=\id_{P}$ modulo isotopy rel.\ $P$. Then $G_P$ is the group of deck transformations of the universal cover $\pi$.

Let $\phi\colon (\Sp,P)\to \CDach$ be an orientation-preserving homeomorphism. Then there exist an orientation-preserving homeomorphism $\widetilde\phi\colon(\Sp,P)\to \CDach$ and a rational map $f_\phi\colon \CDach\to\CDach$ such that the following diagram commutes:
\begin{equation}\label{eq:Thurston}
\xymatrix{(\Sp,P) \ar[r]^{\widetilde{\phi}} \ar[d]_f
&{(\widehat{\mathbb{C}},\widetilde{\phi}(P))} \ar[d]^{f_\phi} \\ (\Sp,P) \ar[r]^{\phi}
&{(\widehat{\mathbb{C}},\phi(P))}\rlap{.}}
\end{equation} 
Namely, the map $\widetilde\phi$ is obtained by pulling back the standard complex structure on $\CDach$ by $\phi\circ f$ and afterwards applying the Uniformization Theorem; the map $f_\phi$ is then the composition $\phi\circ f\circ \widetilde{\phi}^{-1}$. The diagram~\eqref{eq:Thurston} induces the (\emph{Thurston}) \emph{pullback map} $\sigma_f\colon \T_P\to\T_P$ given by $[\phi]\mapsto[\widetilde{\phi}]$. It is well-known that $\sigma_f$ is well-defined and holomorphic. 

The pullback map $\sigma_f$ was used by W.~Thurston in his proof of the celebrated \emph{Characterization Theorem of Rational Maps} \cite{Douady1993}. One of the key observations is that the Thurston map $f$ is conjugate to a rational map modulo isotopy (rel.\ $P$) if and only if the pullback map $\sigma_f$ has a fixed point.
When the orbifold $\Orb_f$ is hyperbolic, the fixed point $\tau_f$ of $\sigma_f$ is unique (if it exists), and the iterates $\sigma^n_f(\tau)$ converge to $\tau_f$ as $n\to\infty$ for any starting point $\tau\in \T_P$.

Following \cite{koch2013teichmuller}, the \emph{Hurwitz space} $\W_f$ associated with $f$ is the set of all triples $(f_\phi,\phi|_P,\widetilde\phi|_P)$, where $\phi\colon (\Sp,P)\to \CDach$ is an orientation-preserving homeomorphism and $\widetilde\phi, f_\phi$ are determined by the diagram \eqref{eq:Thurston}, modulo the following equivalence relation: given two triples $(f_1, \iota_1, j_1)$ and $(f_2, \iota_2, j_2)$, where $f_1,f_2$ are rational maps on $\CDach$ and $\iota_1,\iota_2,j_1,j_2$ are injections from $P$ to $\CDach$, we declare $(f_1, \iota_1, j_1) \sim (f_2, \iota_2, j_2)$ if there exists two M\"obius transformations $M_\iota$ and $M_j$ such that 
\[(f_2,\iota_2,j_2)= (M_\iota\circ f_1 \circ M_j^{-1}, \, M_\iota\circ  \iota_1,\, M_j \circ j_1).\]
The Hurwitz space $\W_f$ admits a natural complex analytic structure and is isomorphic (as a complex manifold) to 
the quotient of $\T_P$ by the \emph{group of liftables} $H_f$. The latter is a finite index subgroup of  the pure mapping class group $G_P$ given by
\[H_f:=\{[h]\in G_P\colon  \text{there exists } [\widetilde h]\in G_P \text{ such that } h \circ f = f \circ \widetilde h\}.\]

The induced map $\omega_f\colon \T_P\to \W_f$ defined by 
\[
[\phi]\mapsto[(f_\phi,\phi|_P,\widetilde\phi|_P)]
\] is a well-defined holomorphic covering map. Furthermore, we have two holomorphic maps $X,Y\colon \W_f\to \M_P$ given by 
\[[(f_\phi,\phi|_P,\widetilde\phi|_P)]\mapsto[\widetilde\phi|_P] \text{\quad and \quad} [(f_\phi,\phi|_P,\widetilde\phi|_P)]\mapsto[\phi|_P],\]
respectively, such that the following diagram commutes
\begin{align}\label{eq:Wspace}
\xymatrix{ & \T_P\ar[dd]_{\pi}\ar[rr]^{\sigma_f} \ar[dr]^{\omega_f} & &
\T_P \ar[dd]^{\pi} \\ && \W_f\ar[dl]_Y \ar[dr]^X &\\ & \M_P & & \M_P.}
\end{align}
Actually, $Y$ is a covering map of finite degree. 

By  work of Selinger \cite{selinger2012thurston}, the Thurston pullback map $\sigma_f$ extends continuously to the completion $\overline{\T_P}$ of the Teichm\"uller space $\T_P$ with respect to the Weil--Petersson metric. We denote the corresponding extension also by $\sigma_f\colon \overline{\T_P}\to \overline{\T_P}$ for simplicity. 
Since the action of $G_P$ on $\T_P$ extends continuously to $\overline{\T_P}$ as well, we obtain the following commutative diagram extending \eqref{eq:Wspace}:
\begin{align}\label{eq:WspaceBar}
\xymatrix{ & \overline{\T_P} \ar[dd]_{\pi}\ar[rr]^{\sigma_f} \ar[dr]^{\omega_f} & &
\overline{\T_P} \ar[dd]^{\pi} \\ && \overline{\W_f}\ar[dl]_{Y} \ar[dr]^{X} &\\ & \overline{\M_P} & & \overline{\M_P},}
\end{align}
where $\overline{\W_f}:=\overline{\T_P}/H_f$ and $\overline{\M_P}:=\overline{\T_P}/G_P$. 
%
%

From now on we specialize to the case when $|P|=4$. In this case, up to M\"obius conjugation, we may assume that $P=\{0,1,\infty,z_0\}$ with $z_0\in\CDach\setminus \{0,1,\infty\}$. Then we have natural identifications 
$\T_P=\Halb=\{z\in \C: \Im(z)>0\}$ and $\M_P = \CDach \setminus \{0,1,\infty\}$  such that the universal covering map $\pi\colon \T_P \to \M_P$ can  be  described in the following way. Let $\Omega$ be the open ideal hyperbolic  triangle in $\Halb$ with vertices at $0,1,\infty$. We define $\pi$ on $\Omega$ to be the Riemann map that sends $\Omega$ onto $\Halb\subset \M_P$ so that the homeomorphic extension of $\pi$ to $\overline {\Omega}$ satisfies $\pi(0)=0$, $\pi(1)=1$, and $\pi(\infty)=\infty$. We then extend $\pi$ to the whole upper half-plane by reflection. It follows that the group $G_P$ of deck transformations of $\pi$ is freely generated by the M\"obius transformations $z\mapsto z+2$ and $z\mapsto\frac{z}{-2z+1}$. Furthermore, the Weil--Petersson completion $\overline{\T_P}$ is given by $\Halb\cup \widehat{\Q}$, where $\widehat{\Q}=\Q\cup \{\infty\}$ denotes the set of extended rational numbers. In addition,  $\overline{\M_P}=\CDach$ and $\overline{\W_f}$ is a closed Riemann surface of finite genus. 
By Riemann's theorem on removable singularities, the continuous extensions
\begin{equation}\label{eq:mod_corr}
    X, Y\colon \overline{\W_f}\to\overline{\M_f},
\end{equation}
of the holomorphic maps $X, Y\colon \W_f \to \M_f$
are holomorphic as well.
We call \eqref{eq:mod_corr} the (\emph{extended}) \emph{moduli space correspondence} associated with $f$.

For $|P|=4$ there is a canonical bijective correspondence between the 
 \emph{Weil--Petersson boundary} $\partial \T_P:=\overline{\T_P}\setminus \T_P = \widehat{\Q}$ and  the set $\mathscr{C}_P$ of isotopy classes of essential (non-oriented) simple closed curves in $\Sp\setminus P$. Namely, each point $\tau \in 
  \T_P=\Halb$ corresponds to a unique conformal structure on the sphere $S^2$ and hence to  a unique complete hyperbolic metric on the punctured sphere 
  $S^2\setminus P$. As $\tau\in \Halb$ approaches a point $r\in \partial \T_P$,
  there exists a unique isotopy class  $[\gamma]\in \mathscr{C}_P$ such that 
  the length of the unique hyperbolic geodesic in  $[\gamma]$ approaches $0$. 
 This gives a bijective correspondence  $r\in \partial \T_P\longleftrightarrow [\gamma]\in \mathscr{C}_P$. 
%

Let us set $\overline{\partial T_P}:=\partial\T_P \cup \{\trivial\}$, where $\trivial$ represents the interior of $\overline{\T_P}$. Then the extension $\sigma_f\colon \overline{\T_P}\to \overline{\T_P}$ of the Thurston pullback map  induces a natural \emph{pullback map} $\partial\sigma_f$ on $\overline{\partial T_P}$: given $r\in \widehat{\Q}=\partial\T_P$, we set $\partial\sigma_f(r) = \sigma_f(r)$ if $\sigma_f(r)\in \partial \T_P$, and $\partial\sigma_f(r)= \trivial$ if  $\sigma_f(r)\in \T_P$; in addition, we also set $\partial\sigma_f(\trivial)=\trivial$. It then immediately follows from \cite[Prop.~6.1]{selinger2012thurston} that the following diagram commutes: 
\begin{equation}\label{eq:curves-and-boundary}\xymatrix{\overline{\mathscr{C}_P}=\mathscr{C}_P\cup\{\trivial\} \ar[d]_{\mu_f} \ar@{<->}[r] 
&\overline{\partial \T_P}= \partial\T_P \cup\{\trivial\}  \ar[d]_{\partial\sigma_f}
\\ \overline{\mathscr{C}_P}=\mathscr{C}_P\cup\{\trivial\} \ar@{<->}[r]
&\overline{\partial \T_P}= \partial\T_P \cup\{\trivial\}.}
\end{equation}
Here, $\mu_f\colon \overline{\mathscr{C}_P}\to \overline{\mathscr{C}_P}$ is the pullback map on $\overline{\mathscr{C}_P}$ defined in the introduction and $\overline{\mathscr{C}_P}\longleftrightarrow\overline{\partial \T_P}$ represents the obvious extension of the canonical one-to-one correspondence between $\mathscr{C}_P$ and $\partial \T_P$.

\section{Tile covers and  global attractors}\label{sec:FGA-invariant-tiling}

We now provide the proof of Theorem~\ref{thm-intro:FGA-tiling} stated in  the introduction. It  uses  diagram \eqref{eq:curves-and-boundary} relating the pullback maps $\mu_f$ and $\partial\sigma_f$. 
%

 
\begin{proof}[Proof of Theorem~\ref{thm-intro:FGA-tiling}]  Let $\mathcal{U}$ be a cover of $\overline{\T_P}= \Halb^*$
by some of its subsets  as in the statement, and $r_0\in \partial\T_P = \widehat{\Q}$ be an arbitrary
 point in the Weil--Petersson boundary. We then define a sequence $\{T_n\}_{n\in \N_0}$  of sets  $T_n\in\mathcal{U}$ inductively as follows. For  $T_0$ we choose 
  a set  in  $\mathcal{U}$ with $r_0\in T_0 \cap \partial\T_P$.
  When $T_n\in \mathcal{U}$ for some $n\in \N_0$ has been selected, by hypothesis  \ref{item:ii}
  we can find a set $T_{n+1}\in \mathcal{U}$ such that $\sigma_f(T_n)\sub T_{n+1}$, 
  providing the inductive step in the definition of  $\{T_n\}_{n\in \N_0}$. 
  
We denote by $\tau_f\in \T_P=\Halb$ the unique fixed point of $\sigma_f$. By
\ref{item:i} we  can then choose a neighborhood $U\sub \Halb$ of $\tau_f$ that meets only finitely many 
sets in $\mathcal{U}$.   Let  $\NC$ be the finite family of all sets $T\in \mathcal{U}$ with $T\cap U\ne \emptyset$. 
 
    \begin{claim} We have $T_n \in \NC$ for all sufficiently large $n$.
    \end{claim}

    Indeed, by \ref{item:ia} we can choose a point $\tau_0 \in T_0\cap \T_P$. We then    set $\tau_n:= \sigma^n_f (\tau_0)$ for $n\in \N_0$. The choice  of $\{T_n\}_{n\in \N_0}$ and  induction imply that 
     $\tau_n\in T_n$  for all $n\in \N_0$. Since $\tau_n \to \tau_f$ as $n\to \infty$, 
 we have     $T_n\cap U\ne \emptyset$ and so $T_n\in \NC$ for all sufficiently large $n$, proving the Claim.

  \smallskip        
By construction, we also have $\sigma_f^n(r_0) \in T_n$ for all $n\in\N_0$.  This and the Claim imply that under iteration of $\sigma_f$ the point $r_0\in \partial\T_P $ eventually lands in $\T_P$  or in the 
  set  $\bigcup_{T\in \NC} (T\cap \partial \T_P)$. Since the  latter  set does not depend on $r_0$ and is finite by  \ref{item:ia},  the pullback map $\partial \sigma_f$ on $\overline{\partial T_P}$ has a finite global attractor. 
  The desired statement  now immediately follows from diagram~\eqref{eq:curves-and-boundary}. 
\end{proof}

\begin{remarks}\mbox{}
\begin{enumerate}[label=\arabic*), leftmargin=*]
\item If $\tau_f$ is the unique fixed point of $\sigma_f$ in $\T_P$ and $U$ a sufficiently small neighborhood of $\tau_f$,  then the  proof of the theorem shows that the set 
\[\{\text{$\partial T\cap \partial \T_P$: $T\in \mathcal{U}$ and
 $ T\cap U\ne \emptyset$}\}\cup \{\trivial\}.\]
 is finite and contains the  global attractor of 
 $\partial\sigma_f$.
 
\item The proof also shows that conditions~\ref{item:i} and~\ref{item:ia} in 
 Theorem~\ref{thm-intro:FGA-tiling} can be relaxed as follows:
 \begin{enumerate}[label=\normalfont{(\roman*')}]

\smallskip
 \item\label{item:i'} the 
 unique fixed point of $\sigma_f$ in $\T_P$ has a neighborhood that meets only finitely many  sets  in $\mathcal{U}$, and each of these sets contains 
  at most finitely many points in $\partial \T_P$, 

\smallskip 
    \item\label{item:ia'} every set in $\mathcal{U}$  contains at least one point in 
    $\T_P$.
    \end{enumerate}

\end{enumerate}    
\end{remarks}

\section{Applications and examples}\label{sec:examples-and-applications}

In this section, we present various applications of Theorem~\ref{thm-intro:FGA-tiling}. We start with providing a sufficient condition for the existence of a $\sigma_f$-invariant cover of the Weil--Petersson completion using the extended moduli space correspondence. Under suitable assumptions, we will obtain 
the sets in our desired cover of $\overline{\T_P}$ from subsets of  moduli space $\M_P$
by lifting them by  the universal covering map $\pi\colon \T_P \to \M_P$ (and afterwards taking the closures in $\overline{\T_P}$).
In order to facilitate this lifting, we have to impose some topological conditions 
on the involved sets. We start with some relevant definitions.

A \emph{tile} $T$ in a topological surface 
$\X$  is  the closure $T=\overline {\Omega}$ of a simply connected region 
$\Omega\sub \X$ with connected and locally connected boundary $\partial \Omega$.
If for the surface  some natural completion $\overline {\X}$ is under consideration 
(such as $\overline{\X}=\Halb^*$ for $\X=\Halb$), then we always consider the closure and boundary of $\Omega \sub \X$ in  $\overline {\X}$. In this case, the tile $T=\overline{\Omega}$ is a subset of $\overline {\X}$.
 We define  $\inter^*(T):=\Omega$ and 
$\partial^* T:=\partial \Omega$ as distinguished subsets of the tile $T$.  
In general, $\inter^*(T)$ and $\partial^* T$ will be different from the interior and boundary of $T$ considered as a subset of the ambient space (this is the reason we use a $*$ in our notation). 

Let $\D=\{z\in \C: |z|<1\}$ denote the open unit disk. If $T$ is a tile in $\X$, then there exists  a continuous surjective map $\eta\colon \overline {\D}\ra T$  such that the restriction $\eta|_\D$ is a homeomorphism  of $\D$ onto  $\inter^*(T)$ and 
$\eta(\partial \D)=\partial^* T$. We call such a map $\eta$ a {\em parametrization} of the tile $T$. Conversely, if $\X$ is a surface with completion $\overline{\X}$ and $\eta\: \overline {\D}\ra \overline{\X}$ is a continuous map such that $\eta|_\D$ is a homeomorphism of $\D$ onto $\eta(\D)$, then 
 $T:=\eta(\overline {\D})$ can be considered as a tile in $\X$ with $\inter^*(T)=\eta(\D)$ and $\partial^* T=\eta(\partial \D)$.   


A \emph{tile cover} $\mathcal{U}$ of $\overline{\X}$ is a cover of $\overline{\X}$ given by tiles. 
We call such a cover $\mathcal{U}$ a {\em tessellation} of $\overline{\X}$ if no two distinct tiles in $\mathcal{U}$  have common interior points, that is, if $T,T'\in 
\mathcal{U}$ and $T\ne T'$, then $\inter^*(T)\cap \inter^*(T')=\emptyset$.

In the following, we again suppose that $f\colon (\CDach, P)\to (\CDach, P)$ is a  rational Thurston map with $|P|=4$ 
    and a hyperbolic orbifold. As we discussed, we may assume that  
    $\T_P = \Halb$, $\partial \T_P = \widehat{\Q}$, and 
      $\M_P=\CDach\setminus \Theta$, where  $\Theta\sub \CDach$ and $|\Theta|=3$. As all thrice-punctures spheres are M\"obius equivalent, we 
   could assume that $\Theta=\{0,1,\infty\}$ (as in our discussion in Section~\ref{sec:pullback-map}),    but it is useful to allow more general sets here.
 We also consider the continuous map $\pi\:
  \overline{\T_P}\ra \overline{\M_P}$ obtained 
  from extending the  holomorphic universal covering map $\pi \: \T_P\ra \M_P$.

  Now let  $G$ be a finite connected planar embedded graph in $\CDach=\overline{\M_P}$ with the vertex set $V(G) \supset\Theta$ (that is, $G$ is a finite connected $1$-dimensional CW-complex in $\CDach$ whose set of $0$-cells contains $\Theta$).  Each complementary component of $G$ in $\overline{\M_P}$
  is a simply connected region with connected and locally connected boundary. It follows that the 
  closure of such a complementary component is a tile in $\M_P$, which we call 
  an \emph{$\M$-tile}. It is clear that $\M$-tiles form a tessellation 
  of $\overline{\M_P}$. 
  
   We now  lift this tessellation 
 to $\overline{\T_P}$ by $\pi$. Essentially, we can do this, because $\partial\M_P=\Theta \sub G$ and so 
 $\pi$ is a covering map over each complementary component of $G$ in $\overline{\M_P}$. 
Then, by lifting parametrizations of $\M$-tiles, one can easily show that each complementary component of $\pi^{-1}(G)$ in $\overline{\T_P}$ 
has a closure that is a tile in $\T_P$. We call these sets  {\em $\T$-tiles}. If $T$ is a $\T$-tile, then $M:=\pi(T)$ is an $\M$-tile and $\pi$ sends $\inter^*(T)$ homeomorphically onto 
$\inter^*(M)$.  Based on standard topological lifting arguments one can show 
that  $\T$-tiles form a tessellation $\mathcal{U}$ of  $\overline{\T_P}$. Moreover, the tessellation $\mathcal{U}$ satisfies conditions  \ref{item:i} and \ref{item:ia} in Theorem~\ref{thm-intro:FGA-tiling}. We leave  the details of these arguments to the reader.

We want to know when this tessellation $\mathcal{U}$ also satisfies condition 
  \ref{item:ii}  in Theorem~\ref{thm-intro:FGA-tiling}.
  For this we say that the graph $G$ is \emph{invariant with respect to the extended moduli space correspondence \eqref{eq:mod_corr}}
if $X^{-1}(G)\subset Y^{-1}(G)$. 

\begin{lemma}\label{lem:inv_tiling} If the graph $G$ is invariant with respect to the extended moduli space correspondence \eqref{eq:mod_corr}, then the tessellation 
$\mathcal{U}$  of 
$\overline{\T_P}$ formed by 
$\T$-tiles  is  $\sigma_f$-invariant. 
\end{lemma}

\begin{proof} Let $T$ be an arbitrary $\T$-tile. Then $M:=\pi(T)$ is an $\M$-tile
with $\pi(\inter^*(T))= \inter^*(M)$.  
By the diagram \eqref{eq:WspaceBar}, the set $\omega_f(\inter^*(T))$ is  connected 
and contained in  $Y^{-1}(\inter^*(M))\sub \overline{\W_f}\setminus Y^{-1}(G)$.   Since $X^{-1}(G)\subset Y^{-1}(G)$, it follows that $\omega_f(\inter^*(T)) \sub  \overline{\W_f}\setminus X^{-1}(G)$.  This in turn implies that 
\[
(\pi\circ \sigma_f)(\inter^*(T))=(X\circ \omega_f)(\inter^*(T))\sub  \overline{\M_f}\setminus G\]
  is a connected set in the complement of $G$. Therefore, there exists an $\M$-tile 
$M'$ with  $(\pi\circ \sigma_f) (\inter^*(T))\sub \inter^*(M')$. Lifting this inclusion by $\pi$, we see that the connected set $\sigma_f(\inter^*(T))$ must be contained in some $\T$-tile $T'$ with $\pi(T')=M'$.
Then also $\sigma_f(T)\sub T'$ and the statement follows.  
\end{proof}

Similarly, we provide a sufficient condition for the existence of a $\sigma_f$-invariant tile cover of $\overline {\T_P}$ arising from a tile cover of $\overline{\M_P}$. We omit the proof as it is completely analogous to the proof of Lemma~\ref{lem:inv_tiling}. 

\begin{lemma}\label{lem:inv_covering2} Suppose there is a finite tile cover $\mathcal{V}$ of $\overline{\M_P}$ such that  every tile $M\in \mathcal{V}$ 
 satisfies the following two conditions:
\begin{enumerate}[label=\normalfont{(\roman*)}]
    \item\label{cov1'} $\inter^*(M) \cap \Theta =\emptyset$,  
\smallskip
    \item \label{cov2'} if $C$ is a component of $Y^{-1}(\inter^*(M))$,
    then $X(C)\sub \inter^*(M')$ for some tile 
    $M'\in  \mathcal{V}$.
\end{enumerate}
Then $\overline{\T_P}$ admits a tile cover satisfying the conditions in Theorem~\ref{thm-intro:FGA-tiling}. 
\end{lemma}

%
%
%

In the rest of the paper, we provide various families (as well as explicit examples) of rational Thurston maps that admit an invariant graph with respect to the extended moduli space correspondence, and thus enjoy a $\sigma_f$-invariant tessellation of the Weil--Petersson completion of  Teichm\"uller space. 

\subsection{Rational Thurston maps with marked critical points}\label{subsec:four-periodic-crit} 
In this subsection, we prove Theorem~\ref{thm:inv-tiling-crit-points}.

\begin{proof}[Proof of Theorem~\ref{thm:inv-tiling-crit-points}] 
 Let $f\colon (\CDach, P)\to (\CDach, P)$ be a rational Thurston map 
 as in the statement. Since $f(P)\sub P$, the  assumption that every point in $P$ is periodic  is equivalent to $f(P)=P$.    

To  describe the moduli space correspondence $X,Y\colon \W_f\to \M_P$ in this case, we fix a subset $\Theta\subset P$ with $|\Theta| = 3$, so that $\M_P=\CDach\setminus \Theta$. We denote by  $p\in \CDach$  the unique point in $P\setminus \Theta$.
Let $\phi\colon (\CDach,P)\to (\CDach, \phi(P))$ be an orientation-preserving homeomorphism normalized so that $\phi|_\Theta=\id_\Theta$. Then there is a unique orientation-preserving homeomorphism $\widetilde{\phi}\colon (\CDach,P)\to (\CDach,\widetilde{\phi}(P))$ with $\widetilde \phi|_\Theta = \id_\Theta$  such that the map $f_\phi:=\phi\circ f\circ {\widetilde\phi}^{-1}\colon \CDach\to\CDach$ is rational. It follows that from the maps $\phi,f, \widetilde{\phi}, f_\phi$ we obtain a  commutative diagram as in \eqref{eq:Thurston} (with $S^2=\CDach)$, and thus $\sigma_f([\phi])=[\widetilde{\phi}]$. 

Set $x:=\widetilde{\phi}(p)$ and $y:=\phi(p)$. Then $\pi([\phi])=y$ and $\pi([\widetilde\phi])=x$. Furthermore, with respect to our normalizations, a point in $\W_f$ is represented by the triple $(f_\phi, y, x)$, and the maps $X$ and $Y$ send this point to $x$ and $y$, respectively.
Note that, by our assumption $C_f\sub P$,  the rational map $f_\phi$ satisfies 
$C_{f_\phi}= \widetilde{\phi}(C_f)\subset \Theta\cup \{x\}$ and 
\[ f_\phi(\Theta\cup\{x\}) =(f_\phi\circ  \widetilde{\phi})(P)=(\phi \circ f)(P)=\phi(P)= \Theta\cup \{y\}.
\] 

Let $G$ be the circle in $\CDach$ with $\Theta\subset G$, which we view as a connected planar embedded graph with $V(G)=\Theta$ and exactly three edges. 
We now use the following result of Eremenko and Gabrielov: \emph{if all critical points of a rational function lie on a circle in the Riemann sphere, then the function maps this circle into a circle} \cite{eremenko1999}. 

This result applied to the map $f_\phi$ implies that if $x\in G$ then $y\in G$. Hence, the graph $G$ is invariant 
with respect to the extended moduli space correspondence \eqref{eq:mod_corr}, meaning that  $X^{-1}(G)\sub Y^{-1}(G)$. It then follows from Lemma~\ref{lem:inv_tiling} that the closures of the complementary components 
of  $\pi^{-1}(G)$ give us a $\sigma_f$-invariant tile cover (actually a tessellation, see the remark below) 
$\mathcal{U}$ of $\overline{\T_P}$. \end{proof}

\begin{remark}Since  
 $G$ is   the circle in $\overline{\M_P}\cong \CDach$ passing through the punctures of $\M_P$,  the closures of the complementary components of $\pi^{-1}(G)$ in 
 $\overline{\T_P}=\Halb^*$ are ideal hyperbolic triangles. They form a $\sigma_f$-invariant tessellation of $\overline{\T_P}$. 
\end{remark}

%

By Theorem~\ref{thm-intro:FGA-tiling}, we obtain as an immediate consequence 
that if $f$ is a rational Thurston map as in Theorem~\ref{thm:inv-tiling-crit-points}, then  the pullback relation $\xleftarrow{f}$
 has a finite global attractor. 

   For the latter conclusion we may in fact drop the assumption that 
the points in $P$ are periodic. Indeed, if $P$ contains a point that is not 
periodic, then there exists a point $p\in P$ that has no preimage in $P$. If we set $\Theta:= P \setminus \{p\}$, then  $f(P)\sub \Theta$ and $P\sub f^{-1}(\Theta)$. Moreover, since $C_f\sub P$, we then necessarily have $P_f\sub \Theta$. 

    Now let $\phi$ be a homeomorphism as in the definition of $\T_P$ with $\phi|_\Theta = \id_\Theta$.  
    Then $\id_{\CDach}$  is isotopic to $\phi$  relative to $ \Theta$, because 
     $|\Theta|=3$; see \cite[Lem.~11.11]{BM17}. Since $P_f\sub \Theta$, such an isotopy lifts by $f$ to an isotopy relative to $f^{-1}(\Theta)\supset P$  from  $\id_{\CDach}$  to a homeomorphism   $\widetilde \phi$   (see \cite[Prop.~11.3]{BM17}). From the endpoints of these isotopies  we obtain a diagram as in  \eqref{eq:Thurston} with $f_\phi=f$. Moreover, here $\widetilde \phi$ is isotopic to $\id_{\CDach}$ rel.\ $P$ no matter what $\phi$ is. This shows that $\sigma_f([\phi])=[\id_{\CDach}]$ and thus 
 $\sigma_f$ is a constant map. In this case, the Finite Global Curve Attractor Conjecture is trivially true.    

We derive the following conclusion.

\begin{cor}\label{cor:marked-crit-points-generalized}
   Let  $f\colon (\CDach, P)\to (\CDach, P)$ be a rational Thurston map with a hyperbolic orbifold , $|P|=4$, and $C_f\subset P$. Then the pullback relation $\xleftarrow{f}$ on curves has a finite global attractor. 
\end{cor}

\subsection{Rational Thurston maps with a moduli space map}\label{subsec:injective-X} 
In this subsection, we consider the special case when the map $X\colon \W_f\to \M_P$ in the diagram \eqref{eq:Wspace} is injective. In this case, we
have that the extension $X\colon \overline{\W_f}\to \overline{\M_P}$ is an injective holomorphic map between (closed) Riemann surfaces, and thus it must be a biholomorphism. It follows that we may consider the \emph{moduli space map} \[g_f:=Y\circ X^{-1}\colon \overline{\M_P}\to \overline{\M_P}.\]
Note that $g_f$ is a rational map on $\CDach$. Furthermore, since $X\colon \W_f\to \M_P$ is injective and $Y\colon \W_f\to \M_P$ is a covering map, we have the inclusions $P_{g_f}\subset \Theta$ and $g_f(\Theta)\subset \Theta$. Hence, $g_f\colon (\CDach, \Theta)\to (\CDach, \Theta)$ is a rational Thurston map.

We can now use existence results for invariant graphs (as provided by \cite[Thm.~1.1]{cui2022graphs} and \cite[Thm.~15.1]{BM17})  to conclude that, for sufficiently large $k\geq 1$, there is a finite connected planar embedded graph $G$ in $\CDach=\overline{\M_P}$ with $V(G)\supset \Theta$ and $g_f^{-k}(G) \supset G$. Since $\sigma_{f^k}=(\sigma_f)^k$, it easily follows that the graph $G$ is invariant under the extended moduli space correspondence for $f^k\colon (\CDach, P)\to (\CDach, P)$. Combining this with Lemma~\ref{lem:inv_tiling} and Theorem~\ref{thm-intro:FGA-tiling} we deduce that the Finite Global Curve Attractor Conjecture holds when $X$ is injective.

\begin{cor}\label{cor:injective-X}
    Suppose $f\colon (\CDach, P)\to (\CDach, P)$ is a rational Thurston map with $|P|=4$, hyperbolic orbifold, and injective $X$.  Then the pullback relation $\xleftarrow{f}$ on curves has a finite global attractor. 
\end{cor}

\subsection{Examples}\label{subsec:examples} 

We now discuss some families and explicit examples of rational Thurston maps that satisfy the assumptions of Theorem~\ref{thm:inv-tiling-crit-points} and Corollary~\ref{cor:injective-X}.

We start with the rational Thurston map $f\colon (\CDach,P)\to(\CDach,P)$ given by $f(z)=\frac{3z^2}{2z^3+1}$ with $P=P_f$. We note that this example has been studied in \cite{lodge2012boundary}, where it was shown that the pullback relation $\xleftarrow{f}$ on curves has a finite global attractor using algebraic techniques. A straightforward computation shows that $C_f=\{0,1,\omega,\bar\omega\}$, where $\omega:=-1/2+i\sqrt{3}/2$ is a cube root of unity, and that $f(0)=0$, $f(1)=1$, $f(\omega)=\bar \omega$, and $f(\bar \omega)=\omega$.  That is, $P_f=\{0,1,\omega,\bar\omega\}$, and $f$ has the following ramification portrait:

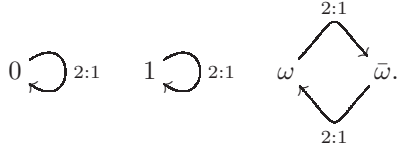
\[
\xymatrix{
    0\ar@(ru,rd)[]^{2:1}
  }
  \quad   
  \xymatrix{
    1\ar@(ru,rd)[]^{2:1}
  }
  \quad  
  \xymatrix{
    \omega \ar@(ru,lu)[r]^{2:1} & \bar\omega \ar@(ld,rd)[l]^{2:1}\rlap{.}
  }
\]
In particular, the map $f$ satisfies the assumptions of Theorem~\ref{thm:inv-tiling-crit-points}. However, in this case, we can show the existence of an invariant connected graph with respect to the extended moduli space correspondence directly, without relying on \cite{eremenko1999}.

Set $\Theta:=\{1,\omega,\bar\omega\}$. The respective correspondence \eqref{eq:mod_corr} for $f$ is explicitly computed in \cite[Sec.~4]{Buff2009Pullback}. Namely, it is shown that $\mathcal{W}_f=\CDach\setminus \{1,\omega,\bar\omega,-1,-\omega,-\bar\omega\}$, and that the holomorphic maps    $X,Y \colon \overline{\W_f}\to\overline{\M_P}$ are given by
\[X(\alpha)=\alpha^2 \quad \text{and} \quad Y(\alpha)= \frac{\alpha(\alpha^3 +2)}{2\alpha^3 + 1},\]
respectively. 

Note that we may view the unit circle $\{z\colon |z|=1\}$ as a connected planar embedded graph $G$ in $\overline{\M_P}=\CDach$ with $V(G)=\Theta$ and exactly three edges. Clearly, $X^{-1}(G)=G$. Furthermore, since $\frac{2\alpha^3 +1}{\alpha^3 + 2}$ is a finite Blaschke product, we have that $Y(G)\subset G$. It follows that the graph $G$ is invariant with respect to the extended moduli space correspondence \eqref{eq:mod_corr}, and thus the pullback map $\partial\sigma_f$ on $\overline{\partial\T_P}$ has a finite global attractor by Theorem~\ref{thm-intro:FGA-tiling}. To say precisely what the global attractor of $\partial\sigma_f$ is in $\overline{\partial \T_P}=\widehat{\Q}\cup \{\trivial\}$, we normalize $\T_P=\Halb$ and $\pi\colon \T_P\to \M_P$ so that the fixed point $\tau_f$ is in the open ideal hyperbolic triangle $\Omega \subset \Halb$ with vertices at $0,1,\infty$, and $\pi$ sends $\Omega$ onto the unit disc $\D\subset\M_P$ with $0,1,\infty$ mapping to $\bar{\omega}, 1, \omega$, respectively. Since $\D$ is an $\M$-tile and $\pi(\tau_f)=0\in \D$, we may conclude that the attractor is contained in the set $\NC:=\{0,\,1,\,\infty,\,\trivial\}\sub \widehat{\Q}\cup \{\trivial\}$. A straightforward computation shows that $1\mapsto 1$ and $0\mapsto \infty \mapsto 0$, and thus $\NC$ is the global attractor of $\partial\sigma_f$. 


Further examples of rational Thurston maps satisfying the assumptions of Theorem~\ref{thm:inv-tiling-crit-points}  are provided by critically fixed rational maps (with four critical points) studied in \cite{hlushchanka2019tischler}, as well as the complex conjugates of critically fixed anti-rational maps with real coefficients  (and four critical points) studied in \cite{geyer2020classification}. The former include the rational map $f(z)=\frac{3z^5+5z}{5z^4+3}$, for which the Hurwitz space $\W_f$ is a (punctured) torus.  

Let us now consider the \emph{rabbit polynomial}, which is the quadratic polynomial of the form $z^2+c$ whose unique finite critical point $0$ lies in a $3$-cycle and with $c\approx -0.12256+0.74486i$. Conjugate  the rabbit polynomial by an affine map sending $0,c,\infty$ to $0,1,\infty$, respectively. We get a rational Thurston map $f\colon (\CDach,P)\to  (\CDach,P)$ with $P:=P_f=\{0,1,x,\infty\}$. In fact, $f(z)=cz^2+1$ and $x\approx 0.87744 + 0.74486 i$.

Set $\Theta:=\{0,1,\infty\}$. Following \cite{koch2013teichmuller}, the respective Hurwitz space $\W_f=\CDach\setminus\{0,1,-1,\infty\}$ and the (extended) moduli space correspondence maps are given by 
\[X(\alpha)=\alpha \text{\quad and \quad} Y(\alpha)=1-\frac{1}{\alpha^2}.\]
Clearly, the rational Thurston map $f\colon (\CDach,P)\to  (\CDach,P)$ satisfies the assumptions of Corollary~\ref{cor:injective-X}. In particular, the moduli space map $g_f=Y\circ X^{-1}$ is given by $z\mapsto 1-\frac{1}{z^2}$. It is straightforward to check that $P_{g_f}=\{0,1,\infty\}$ and that the extended real line $G:=\R\cup \{\infty\}$ is forward-invariant under $g_f$. Hence, we may view $G$ as planar embedded graph in $\CDach$ with $V(G)=\Theta$, and the graph $G$ is invariant with respect to the extended moduli space correspondence. 

Let us normalize $\T_P=\Halb$ and $\pi\colon \T_P\to \M_P$ so that the fixed point $\tau_f$ is in the open ideal hyperbolic triangle $\Omega \subset \Halb$ with vertices at $0,1,\infty$, and $\pi$ sends $\Omega$ onto the upper half-plane $\Halb\subset\M_P$ with $0,1,\infty$ mapping to $0, 1, \infty$, respectively. Then $\pi(\tau_f)=x\in\Halb$, and we may conclude that the global attractor of $\partial\sigma_f$ is contained in  the set $\NC:=\{0,1,\infty,\trivial\}\sub  \widehat{\Q}\cup \{\trivial\}$. A straightforward computation shows that $1\mapsto \infty\mapsto 0\mapsto 1$, and thus $\NC$ is the global attractor of $\partial\sigma_f$. 


Further examples of rational Thurston maps satisfying the assumptions of Corollary~\ref{cor:injective-X} are provided by polynomial maps with periodic critical points and four (marked) postcritical points  \cite[Prop.~5.1]{koch2013teichmuller}, uncritical polynomial maps with four (marked) postcritical points  \cite[Prop.~5.3]{koch2013teichmuller}, and rational Thurston maps with three postcritical points and an extra marked fixed point studied in \cite{smith2024curve}; see also \cite{kelsey2019quadratic,prochorov2024finite}.

\bibliographystyle{alpha}
\bibliography{FGA}
\end{document}